\documentclass[12pt]{article}
\usepackage[margin=3.4cm]{geometry}
\usepackage[utf8]{inputenc}
\usepackage[colorlinks=true,citecolor=black,linkcolor=black,urlcolor=blue]{hyperref}
\usepackage{amsthm,amsmath,amssymb}
\usepackage[normalem]{ulem}

\usepackage{pgf,tikz,tkz-graph}
\usetikzlibrary{arrows,shapes}

\def\Chip{{\rm Chip}}
\def\dist{{\rm dist}}
\def\Div{{\rm Div}}
\def\deg{{\rm deg}}
\def\minfas{{\rm minfas}}

\def\rank{{\rm rank}}

\theoremstyle{plain}
\newtheorem{thm}{Theorem}[section]

\newtheorem{lemma}[thm]{Lemma}
\newtheorem{prop}[thm]{Proposition}
\newtheorem{claim}[thm]{Claim}
\newtheorem*{claim*}{Claim}
\newtheorem{cor}[thm]{Corollary}

\theoremstyle{definition}
\newtheorem{defn}[thm]{Definition}

\newtheorem{rem}[thm]{Remark}

\theoremstyle{remark}

\title{Chip-firing based methods in the Riemann--Roch theory of directed graphs}

\author{B\'alint Hujter\thanks{Supported by the Hungarian National Research, Development and Innovation Office – NKFIH, grant no. K109240.}\\
	\small MTA-ELTE  Egerv\'ary  Research  Group, \\[-0.8ex]
	\small Department  of  Operations  Research,\\[-0.8ex]
	\small E\"otv\"os  Lor\'and University,\\[-0.8ex] 
	\small Budapest, Hungary.\\
	\small\tt hujterb@cs.elte.hu\\
	\and
	Lilla T\'othm\'er\'esz\thanks{Supported by the Hungarian National Research, Development and Innovation Office – NKFIH, grant no. K109240.}\\
	\small MTA-ELTE  Egerv\'ary  Research  Group, \\[-0.8ex]
	\small Department  of Computer Science,\\[-0.8ex]
	\small E\"otv\"os  Lor\'and University,\\[-0.8ex] 
	\small Budapest, Hungary.\\
	\small\tt tmlilla@cs.elte.hu\\
}

\begin{document}
	
	\date{}
	
	\maketitle
	
	\renewcommand{\thefootnote}{}
	
	
	
	\renewcommand{\thefootnote}{\arabic{footnote}}
	\setcounter{footnote}{0}

\begin{abstract}
Baker and Norine proved a Riemann--Roch theorem for divisors on undirected graphs.
The notions of graph divisor theory are in duality with the notions of the chip-firing game of Bj\"orner, Lov\'asz and Shor. We use this connection to prove Riemann--Roch-type results on directed graphs.
We give a simple proof for a Riemann--Roch inequality on Eulerian directed graphs, improving a result of Amini and Manjunath. 
We also study possibilities and impossibilities of Riemann--Roch-type equalities in strongly connected digraphs and give examples.
We intend to make the connections of this theory to graph theoretic notions more explicit via using the chip-firing framework.
\end{abstract}

\maketitle

\section{Introduction}

In 2007, Baker and Norine proved a graph-theoretic analogue of the classical Riemann--Roch theorem for algebraic curves \cite{BN-Riem-Roch}. 
This result inspired much research about Riemann--Roch theorems on tropical curves, lattices and directed graphs \cite{lattice,asadi,trop_GK}.

As pointed out already by Baker and Norine, Riemann--Roch theory on graphs is in close relationship with the chip-firing game of Bj\"orner, Lov\'asz and Shor \cite[Section 5.5]{BN-Riem-Roch}. There is a duality between these theories, that enables one to translate the notions of one theory to the language of the other.
This connection has already turned out to be fruitful in both directions.
NP-hardness results in both chip-firing (halting problem of chip-firing games on directed multigraphs \cite{farrell-levine-coeulerian}) and Riemann--Roch theory (the computation of the rank of a divisor \cite{rang_NP_nehez}) have recently been proved by translating the problem to the dual language.
We note that the graph divisor theory of Baker and Norine is sometimes also referred to as chip-firing, but in this paper by chip-firing we always mean the game of Bj\"orner, Lov\'asz and Shor.

This paper attempts to gather the knowledge about the Riemann--Roch type theorems on directed graphs, together with some new results. 
One of our goals is to make the connections of this theory to graph theoretic notions explicit. 
We feel that these connections are more clear-cut in the chip-firing language, hence we chose to use this framework in the paper.

The outline of the paper is the following: In Section \ref{sec:prelim} we introduce the necessary background. In Section \ref{sec:turnback-arc-set} we give a characterization of non-terminating chip-distributions of Eulerian digraphs using turnback arc sets. This characterization is a variant of some previous results \cite{BL92,Perrot,rang_NP_nehez}, designed to suit our purposes in the paper.
We note that some variant of the specialization of the characterization to undirected graphs plays a central role in each proof of the Riemann--Roch theorem for undirected graphs (see \cite{BN-Riem-Roch,cori,wilmes_bsc}). 

Using the characterization of non-terminating distributions, in Section \ref{s:RR-type-ineq}, we prove a Riemann--Roch type inequality for Eulerian digraphs which is the main result of this paper. This inequality was conjectured and proved for a special case by Amini and Manjunath \cite[Section 6.2]{lattice}. The Riemann--Roch theorem of undirected graphs by Baker and Norine also follows from our inequality as a special case.

In Section \ref{sec::RR_general}, we give a necessary and sufficient condition for a strongly connected digraph to have the Riemann--Roch property, and point out that this characterization is equivalent to the abstract Riemann--Roch theorem of Baker and Norine. We also investigate the natural Riemann--Roch property defined by Asadi and Backman, and show that an Eulerian digraph has the natural Riemann--Roch property if and only if it corresponds to an undirected graph.
We also give various examples of directed graphs satisfying or violating Riemann--Roch-type theorems. 

\section{Preliminaries} \label{sec:prelim}

\subsection{Basic notations}

Throughout this paper, \emph{digraph} means a finite directed graph that can have multiple edges but no loops.
The vertex set and edge set of a digraph $G$ are denoted by $V(G)$ and $E(G)$, respectively. 
We always assume our digraphs to be weakly connected, i.e.~the undirected graph obtained by forgetting the orientations is a connected graph.

The number of directed edges from $u$ to $v$ is denoted by $\overrightarrow{d}(u, v)$.
For a vertex $v$, the indegree and the outdegree of $v$ are denoted by $d^-(v)$ and $d^+(v)$, respectively (i.e, $d^-(v) = \sum_{u \in V(G)} \overrightarrow{d}(u, v)$ and $d^+(v) = \sum_{w \in V(G)} \overrightarrow{d}(v, w)$).

Throughout this paper, we think of undirected graphs as special digraphs, replacing each undirected edge $uv$ by a pair of oppositely directed edges. We call a digraph \emph{bidirected} if $\overrightarrow{d}(u,v)=\overrightarrow{d}(v,u)$ holds for each $u,v\in V$. These are exactly the digraphs corresponding to an undirected graph. 

Bidirected graphs form a subclass of Eulerian digraphs. We call a digraph $G$ \emph{Eulerian} if $d^+(v) = d^-(v)$ holds for each vertex $v \in V(G)$. It is well-known that a weakly connected Eulerian digraph is always \emph{strongly connected}, i.e.~it contains a directed path from $u$ to $v$ and a directed path from $v$ to $u$ for every pair of vertices $\{u,v\}$.

A digraph is \emph{acyclic} if it contains no directed cycle.
For an ordering $v_1,v_2,\dots,$ $v_{|V(G)|}$ of $V(G)$, we call an edge from $v_i$ to $v_j$ with $i<j$ a \textit{forward arc}, and an edge from $v_i$ to $v_j$ with $j<i$ a \textit{backward arc} with respect to this ordering. It is well-known that a digraph is acyclic if and only if there is an ordering of its vertices without backward edges. Such an ordering is often called \emph{topological}.

\subsection{Integer vectors and linear equivalence}

This paper focuses on two related fields, graph divisor theory, and the chip-firing game. The basic objects in both fields are integer vectors indexed by the vertices of a (di)graph $G$. We denote the set of such vectors by $\mathbb{Z}^{V(G)}$.

We think of the elements of $\mathbb{Z}^{V(G)}$ in three ways simultaneously:
\begin{itemize}
	\item[] as vectors $f \in \mathbb{Z}^{|V(G)|}$ with coordinates indexed by the vertices of $G$;
	\item[] as functions $f:V(G)\to \mathbb{Z}$;
	\item[] as elements of the free Abelian group on the set of vertices of $G$.
\end{itemize}
\noindent
For any $f \in \mathbb{Z}^{V(G)}$, the \emph{degree} of $f$ is the sum of its coordinates; we denote it by $\deg(f)$, i.e.\ $\deg(f) = \sum_{v \in V(G)} f(v)$. 
We denote by $f\geq g$ if $f(v)\geq g(v)$ $\forall v\in V(G)$.
 
We denote by $\mathbf{0}_G$ ($\mathbf{1}_G$) the vector in $\mathbb{Z}^{V(G)}$ with each coordinate equal to $0$ ($1$). For a vertex $v\in V(G)$, the characteristic vector of $v$ is denoted by $\mathbf{1}_v$, i.e., $\mathbf{1}_v(v)=1$ and $\mathbf{1}_v(u)=0$ for $u\neq v$. We use the notation $\mathbb{Z}^{V(G)}_+=\{x\in \mathbb{Z}^{V(G)}:\ x\geq\mathbf{0}_G \}$.
 
For any digraph $G$, $\mathbf{d}^+_G \in \mathbb{Z}^{V(G)}$ ($\mathbf{d}^-_G \in \mathbb{Z}^{V(G)}$) is the vector with $\mathbf{d}^+_G(v) = d^+(v)$ ($\mathbf{d}^-_G(v) = d^-(v)$) for all $v \in V(G)$.
If there is no danger of confusion, we omit the subscripts.

For any subset $F$ of $E(G)$, the \emph{indegree vector of $F$} is $\mathbf{d}^{-}_F \in \mathbb{Z}^{V(G)}$, where for each $v\in V(G)$, $\mathbf{d}^-_F(v)$ is the number of edges in $F$ with head $v$.

The \emph{Laplacian matrix} of a digraph $G$ is the following matrix $L\in \mathbb{Z}^{V(G)\times V(G)}$:
\[
L(u,v) = \left\{\begin{array}{cl} -d^+(v) & \text{if } u=v \\
        \overrightarrow{d}(v, u) & \text{if } u\neq v.
      \end{array} \right.
\]      

The Laplacian of a digraph $G$ defines an equivalence relation on $\mathbb{Z}^{V(G)}$, that plays a key role in graph divisor theory.
\begin{defn}[Linear equivalence on $\mathbb{Z}^{V(G)}$]
For $x, y \in \mathbb{Z}^{V(G)}$, $x\sim y$ if there exists a $z\in \mathbb{Z}^{V(G)}$ such that $x = y + Lz$.
\end{defn}
Note that this is indeed an equivalence relation. The following claim is straightforward. 

\begin{claim}\label{clm:equi-group-welldefined}
For $x, y \in \mathbb{Z}^{V(G)}$, $x \sim y$ implies $(-x) \sim (-y)$ and $x + a \sim y + a$ for any $a \in \mathbb{Z}^{V(G)}$.
\end{claim}

\subsection{Riemann--Roch theory on directed graphs}

In \cite{BN-Riem-Roch}, Baker and Norine established the Riemann--Roch theory of graphs. They consider undirected graphs, but their basic definitions can naturally be generalized to directed graphs, we give these definitions for digraphs.

Let $G$ be a finite digraph. 
$\Div(G)$ denotes the free Abelian group on the set of vertices of $G$. We identify $\Div(G)$ with $\mathbb{Z}^{V(G)}$. Elements of $\Div(G)$ are called \emph{divisors}. A divisor $f\in\Div(G)$ is called \emph{effective}, if $f\geq \mathbf{0}$. A divisor is called \emph{equi-effective}, if it is linearly equivalent to an effective divisor.

\begin{defn}[The rank of a divisor, \cite{BN-Riem-Roch}] Let $f\in \Div(G)$.
  \begin{equation*}
    \rank(f) = \min\{\deg(g) - 1 : \text{ $g \in \Div(G)$, $g$ is effective, $f-g$ is not equi-effective}\}.
  \end{equation*}
\end{defn}

The Laplacian matrix of an undirected graph is identical to the Laplacian matrix of the corresponding bidirected graph, hence theorems in \cite{BN-Riem-Roch} concerning undirected graphs can 
be translated for bidirected graphs. The central result of Baker and Norine is \cite[Theorem 1.12]{BN-Riem-Roch} which is an analogue of the classical Riemann--Roch theorem. We give it in its equivalent form for bidirected graphs.

\begin{thm}[Riemann--Roch theorem for bidirected graphs, \cite{BN-Riem-Roch}]
  \label{thm::BN}
  Let $G$ be a bidirected graph and let $f$ be a divisor on $G$. Then
  \[
 	\rank(f) - \rank(K_G - f) = \deg(f) - \mathfrak{g} + 1
  \]
  where $\mathfrak{g} = \frac12 |E(G)| - |V(G)| + 1$ and the canonical divisor $K_G$ satisfies $K_G(v)=d^+(v)-2$ for each $v\in V(G)$.
\end{thm}
We refer to $\mathfrak{g}$ as the genus of the digraph. Note that in \cite{BN-Riem-Roch}, $\mathfrak{g}$ is defined as $|E(G)| - |V(G)| + 1$, this difference is explained by the fact that each edge of an undirected graph corresponds to two edges in the bidirected one.

\subsection{Chip-firing}

  Chip-firing is a solitary game on a digraph, introduced by Bj\"orner, Lov\'asz and Shor \cite{BLS91,BL92}. 
In this game we consider a digraph $G$ with a pile of chips on each of its vertices. A position of the game, called a \emph{chip-distribution} (or just distribution) is described by a vector $x \in \mathbb{Z}^{V(G)}$, where $x(v)$ denotes the number of chips on vertex $v \in V$. Note that though originally chip-firing was defined only for non-negative chip-distributions, in this paper, we allow vertices to have a negative number of chips.  We denote the set of all chip-distributions on $G$ by $\Chip(G)$, which we again identify with $\mathbb{Z}^{V(G)}$.   
  
  The basic move of the game is \emph{firing} a vertex. It means that this vertex passes a chip to its neighbors along each outgoing edge, and so its number of chips decreases by its outdegree. In other words, firing a vertex $v$ means taking the new chip-distribution $x + L\mathbf{1}_v$ instead of $x$. Note that $\deg(x+L\mathbf{1}_v)=\deg(x)$.
  
  A vertex $v \in V(G)$ is \emph{active} with respect to a chip-distribution $x$, if $x(v) \ge d^+(v)$.
The firing of a vertex $v \in V(G)$ is \emph{legal}, if $v$ was active before the firing. 
  A \emph{legal game} is a sequence of distributions in which every distribution is obtained from the previous one by a legal firing. 
A legal game terminates if it arrives at a \emph{stable} distribution, which is a chip-distribution without any active vertices.
Chip-firing on an undirected graph is defined as chip-firing on the corresponding bidirected graph.

  The following theorem was proved by Bj\"orner and Lov\'asz.
They state their theorem only for chip-distributions $x \in \Chip(G)$ with $x\geq \mathbf{0}_G$, but it is easy to check that the proof also works for chip-distributions with negative entries.

  \begin{thm}[{\cite[Theorem 1.1]{BL92}}] \label{thm::vegesseg_kommutativ}
    Let $G$ be a digraph and let $x \in \Chip(G)$
    be a chip-distribution. Then starting from $x$, either every legal game can be continued indefinitely, or every legal game terminates after the same number of moves with the same final distribution. Moreover, the number of times a given vertex is fired is the same in every maximal legal game.
  \end{thm}

  Let us call a chip-distribution $x\in \Chip(G)$ \emph{terminating}, if every legal chip-firing game played starting from $x$ terminates, and call it \emph{non-terminating}, if every legal chip-firing game played starting from $x$ can be continued indefinitely. According to Theorem \ref{thm::vegesseg_kommutativ}, a chip-distribution is either terminating or non-terminating.

  One can easily check by the pigeonhole principle that if for a digraph $G$, a distribution $x \in \Chip(G)$  has $\deg(x) > |E(G)| - |V(G)|$ then $x$ is non-terminating (see \cite{BL92}).
  From this it follows that the following quantity, which measures how far a given distribution is from being non-terminating, is well defined.

\begin{defn}[\cite{rang_NP_nehez}]
For a distribution $x \in \Chip(G)$, the \emph{distance} of $x$ from non-terminating distributions is
\[
\dist(x) = \min\{\deg(y) : y \in \Chip(G),\ y\geq \mathbf{0}_G,\ x + y \text{ is non-terminating}\}.
\]
\end{defn}

The following is a useful technical lemma from \cite{BL92}. Since in \cite{BL92} it is only proved for non-negative distributions, but we need it for integer valued distributions, we give a proof. 
\begin{lemma}
	\label{lem::eof_mind_vegt_sokszor_lo}
	On a strongly connected digraph $G$, in any infinite legal game every vertex is fired infinitely often.
\end{lemma}
\begin{proof}
	In a chip-firing game a vertex can only lose 
    chips if it is fired, but if it is fired, it is not allowed to go negative. Hence in a game with initial distribution $x$, on any vertex $v$, the number of chips is always at least $\min\{0,x(v)\}$. Hence the number of chips on any vertex is at most $\sum_{v\in V}\max\{x(v),0\}=:C(x)$ at any time. 
	
	If a legal game is infinitely long, then there is a vertex that fires infinitely often.
	If a vertex is fired infinitely often, then it passes infinitely many chips to its out-neighbors, hence the out-neighbors also need to be fired infinitely often, otherwise they would have more than $C(x)$ chips. By induction, every vertex reachable on directed path from an infinitely often fired vertex is also fired infinitely often. As the digraph is strongly connected, each vertex is reachable on directed path from each vertex, thus every vertex fires infinitely often.
\end{proof}

From Lemma \ref{lem::eof_mind_vegt_sokszor_lo} it follows that for a strongly connected digraph $G$ if $x\in\Chip(G)$ is non-terminating, then playing a legal game, after finitely many steps we arrive at a distribution which is nowhere negative. Hence there exists a non-negative chip-distribution among the non-terminating distributions of minimum degree.
From this, it follows that $\dist(\mathbf{0}_G)$ equals to the minimum degree of a non-terminating distribution on the digraph $G$. 

\subsubsection{Recurrent chip-distributions}

\begin{defn}\cite{godsil-royle}
We call a chip-distribution $x\in \Chip(G)$ \emph{recurrent} if there exists a non-empty legal game that transforms $x$ to itself. Such a game is called \emph{recurring}.
\end{defn}

Recurrent chip-distributions form an important subset of non-terminating chip-distributions. They are studied in the literature (see \cite{godsil-royle,PPW14,Hujter-Kiss-Tothmeresz.17.Reachability} and also \cite{Perrot} in a slightly different model). Here we mention some of their properties needed for the rest of the paper.

\begin{claim}\label{clm:nonterm-leads-to-recurrent}
Recurrent distributions are non-terminating. Any non-terminating chip-firing game leads to a recurrent chip-distribution after finitely many steps.
\end{claim}

\begin{proof}
The first statement is straightforward from Theorem \ref{thm::vegesseg_kommutativ}. For the second statement consider any legal game started from a non-terminating distribution $x$.  
By Lemma \ref{lem::eof_mind_vegt_sokszor_lo} there must be a time when each vertex has fired, and after that time each vertex must have at least $0$ and at most $\deg(x)$ chips at any moment. The number of such distributions is finite hence some chip-distribution must reappear during the game. Such a distribution is recurrent and is reachable from $x$.
\end{proof}

\begin{claim}\label{clm:recurring-games}
In a recurring game on a strongly connected digraph $G$, each vertex must fire. 
\end{claim}
\begin{proof}
Suppose some vertex $v$ does not fire during a recurring game. We can repeat this game indefinitely which results in a non-terminating game contradicting Lemma \ref{lem::eof_mind_vegt_sokszor_lo}.
\end{proof}

\begin{prop}\label{prop:lineq-nonterm-recurrent}
Let $G$ be a strongly connected digraph. A chip-distribution $x \in \Chip(G)$ is non-terminating if and only if there exists some recurrent distribution $x_r \in \Chip(G)$ such that $x_r \sim x$.
\end{prop}

For the proof we need a Lemma, which appeared first in \cite[Lemma 4.3.]{Bond-Levine}. To be self-contained, we give a proof.

\begin{lemma}\label{lemma::termination_equivalence} \cite{Bond-Levine}
Let $G$ be a strongly connected digraph and $x,y \in \Chip(G)$. If $x \sim y$, then $x$ is terminating if and only if $y$ is terminating.
\end{lemma}
\begin{proof}
  By symmetry, it is enough to prove that if $x$ is terminating, then $y$ is also terminating.
  
  Let $x$ be a terminating chip-distribution. Play the chip-firing game starting from $x$ until it terminates. Let the final configuration be $x^*$.
  Clearly, $x^*\sim x\sim y$.
  Let $z\in\mathbb{Z}^{V(G)}$ be the vector with $x^*=y+Lz$. We can suppose that $z\in \mathbb{Z}^{V(G)}_+$, since the Laplacian of a strongly connected digraph has a strictly positive eigenvector with eigenvalue zero \cite[Lemma 4.1]{BL92}. 
  Start a game from $y$ in the following way: If there is an active vertex $v$ that has been fired less than $z(v)$ times, then one such vertex is fired. If there is no such vertex, the game ends.
  Clearly, after at most $\sum_{v\in V(G)} z(v)$ steps, this modified game ends.
  We claim that for the final distribution $y'=y+Lz'$ (where $z'\leq z$), $y'(v)<d^+(v)$ for each vertex $v$.
   Indeed, as the game stopped, for any vertex $v$ with $y'(v)\geq d^+(v)$, $z'(v)=z(v)$. As $x^*$ is stable, $x^*(v)<d^+(v)$.
  But then from $x^*=y'+L(z-z')$ and $z(v)=z'(v)$, we get $d^+(v)> x^*(v)\geq y'(v)$, which is a contradiction.
\end{proof} 

\begin{cor}\label{cor::dist_reprezentasfgtl}
 If $x\sim y$, then $\dist(x)=\dist(y)$.
\end{cor}

\begin{proof}[Proof of Proposition \ref{prop:lineq-nonterm-recurrent}]
The "if" part is guaranteed by Lemma \ref{lemma::termination_equivalence} and the fact that recurrent distributions are non-terminating. The "only if" part is implied by the fact that a non-terminating game leads to a recurrent position (by Claim \ref{clm:nonterm-leads-to-recurrent}).
\end{proof}

\subsection{Duality between chip-distributions and graph divisors}

In this section, we describe the duality between graph divisor theory and the chip-firing game. This duality was first discovered by Baker and Norine \cite{BN-Riem-Roch}. Here we generalize it to directed graphs.

\begin{prop}
    \label{prop::dual_pair}
$x \in \Chip(G)$ is terminating if and only if $\mathbf{d^+}-\mathbf{1}-x$ is equi-effective, 
i.e. there exists $\mathbf{0} \leq y \in \Div(G)$ such that $y \sim (\mathbf{d^+}-\mathbf{1}-x)$. 
\end{prop}
\begin{proof}
On the one hand, if $x$ is terminating then we play the game until it terminates at some chip distribution $x^* \leq \mathbf{d^+}-\mathbf{1}$. Now $\mathbf{0} \leq (\mathbf{d^+}-\mathbf{1}-x^*) \sim (\mathbf{d^+}-\mathbf{1}-x)$.

On the other hand, if for $y \geq \mathbf{0}$, $y \sim (\mathbf{d^+}-\mathbf{1}-x)$, then the chip distribution $x' = \mathbf{d^+}-\mathbf{1} - y$ has no active vertex. As $y \sim (\mathbf{d^+}-\mathbf{1}-x)$, we have $x \sim x'$, and then $x$ is terminating by Lemma \ref{lemma::termination_equivalence}.
\end{proof}

We say that divisor $f \in \Div(G)$ is the \emph{dual pair of} chip-distribution $x \in \Chip(G)$ if they satisfy $f+x =  \mathbf{d^+}-\mathbf{1}$. The following is a straightforward consequence of
Proposition \ref{prop::dual_pair}.

\begin{cor}
    \label{cor::dual_rank=dist}
If divisor $f \in \Div(G)$ is the dual pair of chip-distribution $x \in \Chip(G)$, then
\[
    \rank(f) = \dist(x)-1.
\]
\end{cor}

Corollary \ref{cor::dual_rank=dist} enables us to state equivalent forms of Riemann--Roch-type theorems for graphs using the notion of $\dist$, see Sections \ref{s:RR-type-ineq} and \ref{sec::RR_general}.

\section{Non-terminating chip-distributions and turnback arc sets}
\label{sec:turnback-arc-set}

\begin{defn}
A \emph{feedback arc set} of a digraph $G$ is a set of edges
$F \subseteq E(G)$ such that the digraph $G'=(V(G), E(G) \setminus F)$
is acyclic. We denote
\[
  \minfas(G) = \min\{|F| : F \subseteq E(G) \text{ is a feedback
  arc set}\}.
\]
\end{defn}

It has already been pointed out by Björner--Lovász \cite{BL92} and Perrot--Pham \cite{Perrot} that non-terminating chip-distributions of digraphs are in connection with feedback arc sets.

In this section, our goal is to give a good characterization of non-terminating chip-distributions based on a subclass of feedback arc sets, called turnback arc sets. 

\begin{defn}
A \emph{turnback arc set} of a digraph $G$ is a set of edges
$T \subseteq E(G)$ such that the digraph $G'$ we get by reversing the edges in $T$
is acyclic.
\end{defn}

It is straightforward to check that any turnback arc set is also a feedback arc set. On the other hand, each minimal feedback arc set is also a turnback arc set, which was first observed by Gallai \cite{Gallai}. 
As \cite{Gallai} is not an accessible source, we provide a proof for this fact. 

\begin{prop}\label{prop:tas-fas-orderings}
Let $G$ be a digraph and let $T,F \subset E(G)$.
\begin{enumerate}
	\item[(i)] $T$ is a turnback arc set $\Leftrightarrow$ $T$ is the set of backward edges with respect to some ordering of $V(G)$.
    \item[(ii)] $T$ is a turnback arc set $\Leftrightarrow$ $E(G) \setminus T$ is a turnback arc set.
    \item[(iii)] $F$ is a feedback arc set $\Leftrightarrow$ $F$ contains a turnback arc set as a subset.
\end{enumerate}
\end{prop}

We only prove the $\Rightarrow$ implication of (iii) as all the other parts are straightforward consequences of the definitions and the previous statements. The authors learned the following nice argument from Darij Grinberg.
\begin{proof}[Proof of the $\Rightarrow$ part of (iii)]
Let $F \subseteq E(G)$ be any feedback arc set and let $G'$ denote the digraph that we get from $G$ by deleting the edges of $F$. 
As $G'$ is acyclic, there is a topological ordering $v_1,v_2,\ldots,v_{V(G)}$ with no backward edge in $G'$. Now let $T$ be the set of backward edges in $G$ with respect to this ordering. Clearly, $T \subset F$ and $T$ is a turnback-arc set by part (i) of this proposition.
\end{proof}

\begin{cor}
Let $G$ be a digraph, then 
\[
	\mathcal{E}_{FAS} \supseteq \mathcal{E}_{TAS} \supseteq \mathcal{E}_{mFAS} \supseteq \mathcal{E}_{mcFAS} 	    
\]
where $\mathcal{E}_{FAS}$ denotes the set of feedback arc sets of $G$, $\mathcal{E}_{TAS}$ denotes the set of turnback arc sets of $G$, $\mathcal{E}_{mFAS}$ denotes the set of minimal feedback arc sets of $G$ and $\mathcal{E}_{mcFAS}$ denotes the set of minimum cardinality feedback arc sets of $G$
\end{cor}

\begin{prop}\label{prop:feedback-most-half}
$\minfas(G) \leq \frac12|E(G)|$ holds for any digraph $G$. Equality holds if and only if $G$ is a bidirected graph.
\end{prop}
\begin{proof}
 
Let $F_1$ denote the set of forward arcs and let $F_2$ denote the set of backward arcs with respect to an arbitrary ordering of $V(G)$.
By Proposition \ref{prop:tas-fas-orderings}, $F_1$ and $F_2$ are both feedback arc sets. 
As $E(G)$ is the disjoint union of $F_1$ and $F_2$, at least one of them must contain at most $\frac12 |E(G)|$ edges.

In a bidirected graph, each turnback arc set must contain exactly one version of each bidirected edge. It also follows that a graph $G$ has $\minfas(G)=\frac{1}{2}|E(G)|$ if and only if each turnback arc set has cardinality $\frac{1}{2}|E(G)|$

Suppose that $G$ is not bidirected. Then there is a pair of vertices $(u,v)$ with $\overrightarrow{d}(u,v) \neq \overrightarrow{d}(v,u)$. Consider orderings $u,v,v_3,\dots,v_{n} $ and $v,u,v_3,\dots,v_{n}$. 
The set of forward arcs has different cardinality for these two orderings, implying $\minfas(G) < \frac{1}{2}|E(G)|$
\end{proof}

\subsection{Recurrent chip-distributions and turnback arc sets}

The next lemma and its proof is the detailed version of the "note added in proof" at the end of \cite{BL92}.

\begin{lemma}\label{lemma::recurrent-turnback-general}
Let $G$ be a strongly connected digraph and let $x \in \Chip(G)$ be a  recurrent chip--distribution. Then there exists some turnback arc set $T\subseteq E(G)$ such that $x \geq \mathbf{d}^-_T$.
\end{lemma}
\begin{proof}
Let $x$ be recurrent. Consider any recurring game started from $x$. By Claim \ref{clm:recurring-games}, each vertex must fire at least once during this game.

Let $v_1,v_2, \dots v_{|V(G)|}$ be the ordering of the vertices by the time of their last firing. Let $T$ be the set of backward edges with respect to this ordering. $T$ is a turnback arc set by Proposition \ref{prop:tas-fas-orderings}. 
We show that $x(v_i)\geq d^-_T(v_i)$ for every $i \in \{ 1,2,\ldots,|V(G)| \}$.
After its last firing, $v_i$ had a nonnegative number of chips.
Since then, it kept all chips it received. And as $v_{i+1}, \dots, v_{|V(G)|}$ all fired since the last firing of $v_i$, it received at least
$\displaystyle \sum_{j=i+1}^{|V(G)|} \overrightarrow{d}(v_j,v_i)=d^-_T(v_i)$ chips. So indeed, we have $x(v_i)\geq d^-_T(v_i)$.
\end{proof}

The next lemma, which is a variant of \cite[Lemma 2.4.]{Perrot} shows that for Eulerian digraphs, the converse implication of Lemma \ref{lemma::recurrent-turnback-general} is also true. We note that Lemma 3.2 from \cite{BN-Riem-Roch} (which is a key lemma there) is the special case of the ``if'' direction of the Lemma \ref{lemma::recurrent-turnback-Eulerian} for bidirected graphs, worded in the divisor language. It would also be possible to word Lemma 3.7 in the divisor language.

\begin{lemma}\label{lemma::recurrent-turnback-Eulerian} 
Let $G$ be an Eulerian digraph. A chip-distribution $x \in \Chip(G)$ is recurrent if and only if there exists some turnback arc set $T\subseteq E(G)$ such that $x \geq \mathbf{d}^-_T$.
\end{lemma}
\begin{proof}
Lemma \ref{lemma::recurrent-turnback-general} implies the ``only if'' direction.
For the ``if'' direction, let $T$ be a turnback arc-set satisfying $x \geq \mathbf{d}^-_T$. 
Let $G_T$ be the digraph we get by reversing the edges of $T$. From the definition of turnback arc set, $G_T$ is an acyclic digraph, hence there is a topological ordering $v_1,v_2,\ldots,v_{|V(G)|}$ for $G_T$.
We can legally fire each vertex once according to this ordering. Indeed, let $x_0=x$ and for all $i \in \{1,2,\ldots,|V(G)|\}$, let $x_i$ denote the chip-distribution obtained by firing vertex $v_{i}$ in $x_{i-1}$. In $x_{i-1}$, vertex $v_i$ has all the chips it gathered by the firings of $v_1,\ldots,v_{i-1}$, hence
\[
	x_{i-1}(v_i) = x(v_i) + \sum_{j=1}^{i-1} \overrightarrow{d}(v_j,v_i) 
    \geq 
    d^-_T(v_i) + \sum_{j=1}^{i-1}  \overrightarrow{d}(v_j,v_i) = d^-(v_i) = d^+ (v_i).
\]
We conclude that $v_i$ is active with respect to $x_{i-1}$.

As $G$ is Eulerian, $L \mathbf{1}_G = \mathbf{0}_G$ holds, hence after firing each vertex once, we get back to chip-distribution $x$, proving that $x$ is recurrent.
\end{proof}

Lemmas \ref{lemma::recurrent-turnback-general} and \ref{lemma::recurrent-turnback-Eulerian} with Proposition \ref{prop:lineq-nonterm-recurrent} imply the following theorem, which is the main structural result of this section.

\begin{thm}\label{thm::nonterm_turnback}
Let $G$ be an Eulerian digraph.
A chip-distribu\-ti\-on $x\in \Chip(G)$ is non-terminating if and only if there exists a turnback arc set $T$ of $G$, and a chip-distribution $y\in\Chip(G)$ with $y\geq \mathbf{0}_G$, such that $x \sim \mathbf{d}^-_T + y$.
\end{thm}

\begin{rem}
Lemma \ref{lemma::recurrent-turnback-Eulerian} and Theorem \ref{thm::nonterm_turnback} do not hold for non-Eulerian digraphs. The simplest counterexample is a digraph on two vertices $v_1$ and $v_2$ with two edges from $v_1$ to $v_2$ and one edge from $v_2$ to $v_1$. Then this latter edge forms a turnback arc set in itself, but its indegree vector forms a stable distribution.
\end{rem}
\begin{rem}
Theorem \ref{thm::nonterm_turnback} remains true if one substitutes the term "turnback arc set" with the term "feedback arc set" or "minimal feedback arc set". The reason for choosing "turnback arc set" here is the usefulness of part (ii) of Proposition \ref{prop:tas-fas-orderings}.
\end{rem}

\begin{defn}
Let us call a chip-distribution $x\in\Chip(G)$ \emph{minimally non-ter\-mi\-na\-ting}, if it is non-terminating, but for each $v\in V(G)$, $x-\mathbf{1}_v$ is terminating.
\end{defn}

Theorem \ref{thm::nonterm_turnback} implies that for Eulerian digraphs, the minimally non-terminating chip-distributions are linearly equivalent to the indegree vectors of minimal feedback arc sets. It will turn out in Section \ref{sec::RR_general} that minimally non-terminating distributions are very important from the point of view of Riemann--Roch theorems.

\section{A Riemann--Roch-type inequality for Eulerian digraphs}\label{s:RR-type-ineq}

The following theorem is the main result of this paper.
\begin{thm}[Riemann--Roch inequality for Eulerian digraphs]\label{thm:eulerian-minfas-divisor}
Let $G$ be an Eulerian digraph. Then for any divisor $f$, the inequality 
\[
    \deg(f) - \mathfrak{g}_{\max} + 1
    \leq \rank(f) - \rank(K_G - f) \leq
    \deg(f) - \mathfrak{g}_{\min} + 1
\]
holds with
\[
	\mathfrak{g}_{\max} = |E(G)| - \minfas(G) - |V(G)| + 1 \text{ and } \mathfrak{g}_{\min} = \minfas(G) - |V(G)| + 1.
\]
The canonical divisor $K_G$ satisfies $K_G(v)=d^+(v)-2$ for each $v\in V(G)$.
\end{thm}

\begin{rem}
\begin{enumerate}
	\item Proposition \ref{prop:feedback-most-half} implies that Theorem \ref{thm:eulerian-minfas-divisor} is a generalization of Theorem \ref{thm::BN}.
    \item The theorem is sharp in the sense that for any digraph $G$, there exist divisors $f_{\max}$ and $f_{\min}$ satisfying equalities with $\mathfrak{g}_{\max}$ and $\mathfrak{g}_{\min}$. For details see Remark \ref{rem:minfas-sharpness}.
	\item The bounds of Theorem \ref{thm:eulerian-minfas-divisor} are conjectured by Amini and Manjunath in \cite[Section 6.2.]{lattice}, but are only proved for the special case when each edge has multiplicity at least one by using a limiting argument. For the general case they only obtain weaker bounds.
\end{enumerate}
\end{rem}

We prove Theorem \ref{thm:eulerian-minfas-divisor} by translating it to to language of chip-firing games. By Corollary \ref{cor::dual_rank=dist},  Theorem \ref{thm:eulerian-minfas-divisor} is equivalent to the following statement.

\begin{thm}\label{thm:eulerian-minfas}
Let $G$ be an Eulerian digraph. For each $x\in\Chip(G)$,
\[
\minfas(G)-\deg(x)\leq\dist(x)-\dist(\mathbf{d}^+_G-x)\leq |E(G)|-\minfas(G)-\deg(x).
\]
\end{thm}
\begin{proof}[Proof of Theorem \ref{thm:eulerian-minfas}]
  First, we prove that 
\[  
\dist(\mathbf{d^+}-x) \leq \dist(x) - \minfas(G) + \deg(x).
\]
  From the definition of $\dist(x)$, there exists a chip-distribution $a\geq \mathbf{0}_G$ with $\deg(a)=\dist(x)$ such that $x+a$ is non-terminating.
  By Theorem \ref{thm::nonterm_turnback}, there exists a turnback arc set $T$ of $G$ and a chip-distribution $b \geq \mathbf{0}_G$ such that 
\begin{equation*}
	x+a \sim \mathbf{d}^-_T+b.
\end{equation*}    
Our next goal is to show that $(\mathbf{d^+}-x) + b$ is non-terminating. 
Claim \ref{clm:equi-group-welldefined} implies that 
$\mathbf{d}^+ - x + b \sim \mathbf{d}^+ - \mathbf{d}^-_T + a$. As $G$ is Eulerian 
\[
    \mathbf{d}^+ - \mathbf{d}^-_T + a =
    (\mathbf{d}^- - \mathbf{d}^-_T) + a =
    \mathbf{d}^-_{E(G) \setminus T} + a.
\]
By Proposition \ref{prop:tas-fas-orderings}, $E(G) \setminus T$ is also a turnback arc set. 
We have that 
\[
    (\mathbf{d}^+ - x) + b \sim 
    \mathbf{d}^-_{E(G) \setminus T} + a,
\]
with $a \geq \mathbf{0}_G$. Theorem \ref{thm::nonterm_turnback} gives
that $(\mathbf{d^+}-x) + b$ is non-terminating.
Hence
\[
	\dist(\mathbf{d^+}-x) \leq 
	\deg(b) =
    \deg(x)+\deg(a)-\deg(\mathbf{d}^-_T)
\]
Proposition \ref{prop:tas-fas-orderings} implies $\deg(\mathbf{d}^-_T) \geq \minfas(G)$, which gives the desired lower bound:
\[ 
	\minfas(G)-\deg(x) \leq \dist(x)-\dist(\mathbf{d^+}-x).
\]
  For the upper bound, let $y=\mathbf{d^+}-x$. Then $x=\mathbf{d^+}-y$. From the above argument, we have 
\begin{eqnarray*}  
  \dist(x) ~=~ \dist(\mathbf{d^+}-y)
  &\leq & \dist(y) -\minfas(G) +\deg(y) \\
  &=& 
  \dist(\mathbf{d^+}-x) - \minfas(G) + (|E(G)| - \deg(x)) \\
  &=& \dist(\mathbf{d^+}-x) + |E(G)| - \minfas(G)-\deg(x),
\end{eqnarray*}
giving $\dist(x)-\dist(\mathbf{d^+}-x)\leq |E(G)|-\minfas(G)-\deg(x)$.
\end{proof}
\begin{rem}\label{rem:minfas-sharpness}
Theorem \ref{thm:eulerian-minfas} is sharp in the following sense: for any Eulerian digraph $G$, taking $x$ to be the indegree-distribution of a minimum cardinality turnback arc set, $\dist(x) = \dist(\mathbf{d^+}-x) = 0$, hence 
\[
	\dist(x) - \dist(\mathbf{d^+}-x) = \minfas(G) - \deg(x).
\] 
On the other hand, for $y = \mathbf{d^+}-x$, 
\[
	\dist(y) - \dist(\mathbf{d^+}-y) = |E(G)| - \minfas(G) - \deg(y).
\]    
\end{rem}

\begin{rem}
If we specialize the proof of Theorem \ref{thm:eulerian-minfas} to bidirected graphs and translate it to the language of graph divisors, we get back the proof of Cori and le Borgne for the Riemann--Roch theorem of graphs \cite{cori}.
\end{rem}

\subsection{A non-Eulerian counterexample for the Riemann--Roch inequality}
 
We show by an example that Theorem \ref{thm:eulerian-minfas} does not always hold for non-Eulerian digraphs.
 
\begin{figure}[ht]
\begin{center}
\begin{tikzpicture}
  \GraphInit[vstyle=Dijkstra]
	\Vertex[x=0,y=2,L=$v_1$]{a1}    
	\Vertex[x=2,y=2,L=$v_2$]{a2}    
	\Vertex[x=2,y=0,L=$v_3$]{a3}
	\Vertex[x=0,y=0,L=$v_4$]{a4}
  \tikzset{EdgeStyle/.style={line width=1,->}}
        \Edge(a2)(a3)
  \tikzset{EdgeStyle/.style={bend right=15,line width=1,->}}
    \Edge(a1)(a2)
    \Edge(a2)(a1)
    \Edge(a2)(a4)
    \Edge(a4)(a2)
    \Edge(a3)(a4)
    \Edge(a4)(a3)
    \Edge(a4)(a1)
    \Edge(a1)(a4)
\end{tikzpicture}
\end{center}
\caption{$G_0$, a non-Eulerian digraph where the Riemann--Roch inequality does not hold}\label{fig::RR_ineq_counterex}
\end{figure}
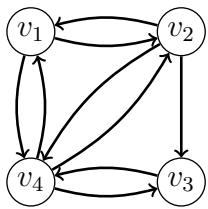

Let $G_0$ be the following digraph (see also Figure \ref{fig::RR_ineq_counterex}).
\begin{eqnarray*}
V(G_0) &=& \{v_1,v_2,v_3,v_4\};\\
E(G_0) &=& \{\overrightarrow{v_1v_2},\overrightarrow{v_1v_4},\overrightarrow{v_2v_1},\overrightarrow{v_2v_3},\overrightarrow{v_2v_4},\overrightarrow{v_3v_4},\overrightarrow{v_4v_1},\overrightarrow{v_4v_2},\overrightarrow{v_4v_3}\}.
\end{eqnarray*}

It is easy to see that $\minfas(G_0)=4$. Indeed, $\minfas(G_0)\geq 4$ as there are 4 edge-disjoint cycles, but it is easy to construct a feedback arc set with 4 edges. Hence 
\begin{eqnarray*}
	\mathfrak{g}_{\max} &=& |E(G)| - \minfas(G) - |V(G)| + 1 = 2;\\
    \mathfrak{g}_{\min} &=& \minfas(G) - |V(G)| + 1 = 1.
\end{eqnarray*}    

One can check by a simple computation that for the Laplacian $L_0$ of $G_0$, the following equation holds for any triplet $(a,b,c)$ of integers
\[
	L_0 \cdot
    \begin{bmatrix} 
    a + 3b + 4c \\ 
    a + 2b + 3c \\
    3a + 6b + 7c \\ 
    2a + 4b + 5c 
    \end{bmatrix} = 
    \begin{bmatrix}
    -2 & 1 & 0 & 1 \\
    1 & -3 & 0 & 1 \\ 
    0 & 1 & -1 & 1 \\
    1 & 1 & 1 & -3 
    \end{bmatrix} \cdot
    \begin{bmatrix} 
    a + 3b + 4c \\ 
    a + 2b + 3c \\
    3a + 6b + 7c \\ 
    2a + 4b + 5c 
    \end{bmatrix} 
    =
    \begin{bmatrix} a \\ b \\ c \\ -(a+b+c) \end{bmatrix}. 
\]
which means that there is only one linear equivalence class of a given degree for $G_0$. Therefore a divisor $f$ is equi-effective if and only if its degree is non-negative, and if $\deg(f) \geq 0$, then $\rank(f) = \deg(f)$.

Now let $f_1=(0,0,0,-1)$ and $f_2=(0,1,-1,2)$ be divisors on $G_0$ (note that $f_1+f_2 = K_{G_0}$). The following computations show that both sides of Theorem \ref{thm:eulerian-minfas-divisor} are violated:
\begin{eqnarray*}
	\rank(f_1) - \rank(K_{G_0} - f_1) = -1 - 2 =
    -3 &<& -2 = \deg(f_1) - \mathfrak{g}_{\max} + 1; \\
    \rank(f_2) - \rank(K_{G_0} - f_2) = 2 - (-1) = 
    3 &>& 2 = \deg(f_2) - \mathfrak{g}_{\min} + 1.
\end{eqnarray*}

\section{Riemann--Roch-type equations for general digraphs}
\label{sec::RR_general}

\subsection{A condition for the Riemann-Roch property in strongly connected digraphs}

It is an interesting question whether a Riemann--Roch type \emph{equation} can be true for general strongly connected digraphs. This question was investigated before in \cite{asadi,lattice}. Here we show that the setting of Riemann--Roch theorems on digraphs is equivalent to the setting of the abstract Riemann--Roch theorem of Baker and Norine \cite[Section 2]{BN-Riem-Roch} and the setting of Amini and Manjunath \cite{lattice}, and give the graphical equivalent of their theorems. This way we can give a more intuitive meaning of the genus. The graphical version is also useful for constructing examples.
 
In the setting of Baker and Norine, $\Div(X)$ is a free Abelian group on a finite set $X$, which is equipped with an equivalence relation satisfying two given properties, (E1) and (E2). These properties are equivalent to the fact that the differences of equivalent divisors form a lattice $\Gamma\subset \mathbb{Z}_0^n$, where $\mathbb{Z}_0^n$ denotes those vectors of $\mathbb{Z}^n$ whose coordinates sum to zero. Amini and Manjunath consider this latter situation, i.e., for them, divisors are elements of $\mathbb{Z}^n$ and for a fixed lattice $\Gamma\subset \mathbb{Z}_0^n$ they call two divisors equivalent if their difference is in $\Gamma$. The case of divisor theory on strongly connected digraphs corresponds to the case if the lattice $\Gamma$ is generated by the Laplacian matrix of a strongly connected digraph. Since by a theorem of Perkinson, Perlman and Wilmes \cite[Theorem 4.11]{PPW14}, each lattice $\Gamma\subset \mathbb{Z}_0^n$ can be generated by the Laplacian matrix of a strongly connected digraph, the graphical case is indeed equivalent to the other two.

Baker--Norine and Amini--Manjunath both obtain a necessary and sufficient condition in their setting for the existence of a Riemann--Roch formula (\cite[Theorem 2.2]{BN-Riem-Roch} and \cite[Theorem 1.4]{lattice}). In this section, we deduce the graphical equivalent of these necessary and sufficient conditions.

We say that a strongly connected digraph $G$ \emph{has the Riemann--Roch property} if there exists some $C\in\Chip(G)$ and integer $t$, such that for each $x\in \Chip(G)$,
\[
	\dist(x)-\dist(C-x)= t-\deg(x).
\]
In this case we say that $C$ is a \emph{canonical distribution} for $G$. Examples of Section \ref{ss::examples} show that such $C$ and $t$ does not always exist for a digraph.

From the divisor-theoretic point of view, the existence of such $C$ and $t$ implies, that for the divisor $K=2\cdot \mathbf{d}^+_G-2\cdot \mathbf{1}_G-C$, each divisor $f\in \Div(G)$ has
\[
	\rank(f)-\rank(K - f)= t-|E(G)|+|V(G)|+\deg(f).
\]

\begin{prop}\label{prop::K_chipszama}
If for a digraph $G$, the Riemann--Roch formula
\[
	\dist(x)-\dist(C-x)= t-\deg(x)
\]
holds for all $x\in \Chip(G)$ for some $C\in\Chip(G)$ and value $t$, then 
\begin{itemize}
	\item[(i)] $\deg(C)=2t$,
	\item[(ii)] $t=\dist(\mathbf{0}_G)$.
\end{itemize}
\end{prop}
\begin{proof}
\emph{(i)} Take an arbitrary $x\in \Chip(G)$, and write up the Riemann--Roch formula for $x$ and for $C-x$. Note that $C-(C-x)=x$.
\begin{eqnarray*}
\dist(x)-\dist(C-x) &=& t-\deg(x) \\
\dist(C-x)-\dist(x) &=& t-\deg(C-x)
\end{eqnarray*}
Summing these two equalities, we get $2t=\deg(C-x)+\deg(x)=\deg(C)$.

\emph{(ii)} Let $\dist(C)=k$. 
The Riemann--Roch formula for $x=\mathbf{0}_G$ says that $\dist(\mathbf{0}_G) = \dist(C) + t - 0 = k+t$. By definition $k$ is non-negative, hence it is enough to show that it is non-positive (i.e.~$C$ is non-terminating).

By the definition of $\dist(\mathbf{0}_G)$, there exists a non-terminating chip-distribution $x_0$ with $\deg(x_0)=\dist(\mathbf{0}_G)=k+t$. The Riemann--Roch formula for $x=x_0$ says that 
\[
\dist(x_0)-\dist(C-x_0)=t-\deg(x_0) = -k.
\]
Since $x_0$ is non-terminating, $\dist(x_0)=0$. Hence we have $\dist(C-x_0)=k$. Using part \emph{(i)}, we have 
\[
\deg(C-x_0) = \deg(C)-\deg(x_0) = 2t - (k + t) = \dist(\mathbf{0}_G)-2k.
\]
As a non-terminating distribution has degree at least $\dist(\mathbf{0}_G)$, we necessarily have $\dist(C-x_0) \geq 2k$, which means $k\geq 2k$. This implies $k=0$.
\end{proof}

\begin{prop} \label{prop::RR=>min_vegtln_deg=dist0}
If the Riemann--Roch formula holds for a digraph $G$, then all minimally non-terminating distributions have degree $\dist(\mathbf{0}_G)$.
\end{prop}
\begin{proof}
Suppose that the Riemann--Roch formula holds for the digraph $G$, $x\in\Chip(G)$ is non-terminating, and $\deg(x)=\dist(\mathbf{0}_G)+k$ with $k>0$. We show that $x$ is not minimally non-terminating.

From the Riemann--Roch formula, $\dist(x)-\dist(C-x)=\dist(\mathbf{0}_G)-\deg(x)$, thus $\dist(C-x)=k$. Since $\deg(C-x)=\dist(\mathbf{0}_G)-k$, this means that there exists a chip-distribution $a\in\Chip(G)$, $a\geq \mathbf{0}_G$, $\deg(a)=k$, such that $C-x+a$ is non-terminating, and is of degree $\dist(\mathbf{0}_G)$.

By the Riemann--Roch formula, $\dist(C-x+a)-\dist(x-a)=\dist(\mathbf{0}_G)-\deg(C-x+a)$. As $\dist(C-x+a)=0$ and $\deg(C-x+a)=\dist(\mathbf{0}_G)$, we have $\dist(x-a)=0$, hence $x-a$ is non-terminating. Since $a\geq \mathbf{0}_G$, and $a\neq \mathbf{0}_G$ we conclude that $x$ is not minimally non-terminating.
\end{proof}

\begin{thm} \label{thm::szuks_elegs_felt_alt_digraf}
Let $G$ be a digraph. $G$ has the Riemann--Roch property if and only if each minimally non-terminating distribution has degree $\dist(\mathbf{0}_G)$, and there exists a distribution $C\in\Chip(G)$ such that for any minimally non-terminating distribution $x\in\Chip(G)$, $C-x$ is also minimally non-terminating. If the above condition holds, then the Riemann--Roch formula holds for $G$ with $C$ as canonical distribution.
\end{thm}

\begin{rem}
As we have already mentioned, Theorem \ref{thm::szuks_elegs_felt_alt_digraf} is equivalent to \cite[Theorem 2.2]{BN-Riem-Roch}. We leave the proof of this equivalence as an exercise for the reader. 
For completeness, we give a proof for Theorem \ref{thm::szuks_elegs_felt_alt_digraf}.  
\end{rem}

\begin{proof}
First we show the ``only if'' direction. Suppose that $G$ has the Riemann--Roch property with canonical distribution $C$. Then by Proposition \ref{prop::RR=>min_vegtln_deg=dist0}, each minimally non-terminating distribution has degree $\dist(\mathbf{0}_G)$. Suppose that $x$ is minimally non-terminating. Then $\dist(x)=0$ and $\deg(x)=\dist(\mathbf{0}_G)$. Since $\dist(x)-\dist(C-x)=\dist(\mathbf{0}_G)-\deg(x)=0$, we have $\dist(C-x)=0$, thus $C-x$ is non-terminating. Also, since $0=\dist(C-x)-\dist(x)=\dist(\mathbf{0}_G)-\deg(C-x)$, we have $\deg(C-x)=\dist(\mathbf{0}_G)$, thus $C-x$ is minimally non-terminating.

Now, we show the ``if'' direction. First note that $\deg(C)=2 \cdot \dist(\mathbf{0}_G)$ is a direct consequence of our conditions. It is sufficient to show that $\dist(x)-\dist(C-x)\geq \dist(\mathbf{0}_G)-\deg(x)$ holds for any $x\in\Chip(G)$. Indeed, by plugging $C-x$ into $x$, and using that $\deg(C-x) = \deg(C) - \deg(x) = 2 \cdot \dist(\mathbf{0}_G)-\deg(x)$, we get
$$
\dist(C-x)-\dist(x) \geq \dist(\mathbf{0}_G)-\deg(C-x) = \deg(x)-\dist(\mathbf{0}_G)
$$ 
which implies the equality.

$x$ is either terminating or non-terminating.

First suppose that $x$ is terminating. Then $\dist(x)=k>0$. This means that there exists $a\in\Chip(G)$, $a\geq \mathbf{0}_G$, $\deg(a)=k$ such that $x+a$ is non-terminating.

There exists at least one minimally non-terminating distribution $y$ such that $x+a=y+b$ where $b \geq \mathbf{0}_G$. (We can take off chips until our distribution gets minimally non-terminating.)

Then by our assumption, $C-y$ is also a minimally non-terminating distribution. We have $C-(x+a)=C-(y+b)$, thus $(C-x)+b=(C-y)+a$. $(C-y)+a$ is non-terminating, as $C-y$ is non-terminating and $a\geq 0$. Thus $\dist(C-x)\leq \deg(b)=\deg(x)+\deg(a)-\deg(y)=\deg(x)+\dist(x)-\dist(\mathbf{0}_G)$.

Hence 
\[
\dist(x)-\dist(C-x)\geq \dist(x)-(\deg(x)+\dist(x)-\dist(\mathbf{0}_G))=\dist(\mathbf{0}_G)-\deg(x).
\]

Now suppose that $x$ is non-terminating. Then there exists a minimally non-terminating distribution $y\in\Chip(G)$ such that $x=y+b$ where $b \geq \mathbf{0}_G$. By our assumptions, $\deg(y)=\dist(\mathbf{0}_G)$ and $C-y$ is also minimally non-terminating.

$C-x+b = C-y$. As $C-y$ is non-terminating, $\dist(C-x)\leq \deg(b)$. On the other hand, $\deg(x)=\deg(y)+\deg(b)=\dist(\mathbf{0}_G)+\deg(b)$. Thus 
\[
\dist(x)-\dist(C-x)\geq 0 - \deg(b) = 0 - (\deg(x)-\dist(\mathbf{0}_G))=\dist(\mathbf{0}_G)-\deg(x).
\]
\end{proof}

\subsection{The natural Riemann--Roch property in Eulerian digraphs}\label{ss::natural_iff_bidirected}

Asadi and Backman introduced the following variant of the Riemann--Roch property \cite[Definition 3.12]{asadi}: 
A digraph $G$ has the \emph{natural Riemann--Roch property}, if it satisfies a Riemann--Roch formula with canonical divisor $K(v)=d^+(v)-2$ for each $v\in V$.
This definition translates to the language of chip-firing in the following way:
\begin{defn}
A digraph $G$ is said to have the natural Riemann--Roch property, if for each $x\in \Chip(G)$: 
\[ 
\dist(x)-\dist(\mathbf{d}^+_G-x)=\frac{1}{2}|E(G)|-\deg(x)
\]
\end{defn}
From the Riemann--Roch theorem for undirected graphs, it follows that each
undirected (that is, bidirected) graph has the natural Riemann--Roch property. However it is left open in \cite{asadi} whether there are any other such graphs.

In Section \ref{ss::example_natural} we show an example that a non-bidirected graph can also have the natural Riemann--Roch property. 
However, the following theorem shows that among Eulerian digraphs only the bidirected graphs have the natural Riemann--Roch property.

\begin{thm}\label{thm:natural_RR}
Let $G$ be an Eulerian digraph. Then $G$ has the natural Riemann--Roch property if and only if it is a bidirected graph corresponding to an undirected graph (i.e.\ $\overrightarrow{d}(u,v) = \overrightarrow{d}(v,u)$ for any pair of vertices $u,v$).
\end{thm}

\begin{proof}
Suppose that $G$ has the natural Riemann--Roch property. Then we have $C(v)=d^+(v)$ $\forall v\in V(G)$, thus, $\deg(C)=|E(G)|$.
Proposition \ref{prop::K_chipszama}
implies that $\dist(\mathbf{0}_G) = \frac{1}{2} \deg(C) =\frac{1}{2} |E(G)|$.
Theorem \ref{thm::nonterm_turnback} says that for Eulerian digraphs, $\dist(\mathbf{0}_G)=\minfas(G)$.
As a consequence, we have $\minfas(G) = \frac{1}{2}|E(G)|$. 
By Proposition \ref{prop:feedback-most-half}, this is only possible if $G$ is bidirected, hence an Eulerian digraph with the natural Riemann--Roch property needs to be bidirected. As we already noted, the Riemann--Roch theorem for undirected graphs implies that bidirected graphs have the natural Riemann--Roch property.
\end{proof}

\subsection{Examples}\label{ss::examples}

In this section we provide examples showing that a digraph may not have the Riemann--Roch property, but for certain digraphs such a theorem can still hold.

\subsubsection{A digraph without Riemann--Roch property}

Consider the following digraph $G_1$ (see also Figure \ref{fig::no_RR}):
\begin{eqnarray*}
V(G_1) &=& \{v_1,v_2,v_3,v_4,v_5,v_6\}, \\
E(G_1) &=& \{\overrightarrow{v_1v_2};\overrightarrow{v_2v_3},\overrightarrow{v_3v_4},\overrightarrow{v_4v_5},\overrightarrow{v_5v_6},\overrightarrow{v_6v_1},\overrightarrow{v_1v_5},\overrightarrow{v_2v_6},\overrightarrow{v_3v_1},\overrightarrow{v_4v_2},\overrightarrow{v_5v_3},\overrightarrow{v_6v_4}\}.
\end{eqnarray*}

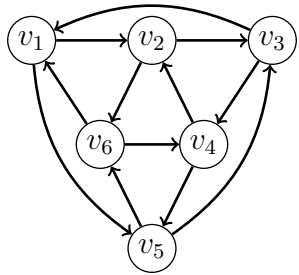
\begin{figure}[ht]
\begin{center}
\begin{tikzpicture}[scale=0.8]
\GraphInit[vstyle=Dijkstra]
	\Vertex[x=0,y=1.73,L=$v_2$]{a1}    
	\Vertex[x=2,y=1.73,L=$v_3$]{a2}    
	\Vertex[x=-2,y=1.73,L=$v_1$]{a3}
	\Vertex[x=-0.86,y=0,L=$v_6$]{a4}
	\Vertex[x=0.86,y=-0,L=$v_4$]{a5}
	\Vertex[x=0,y=-1.73,L=$v_5$]{a6}  
  \tikzset{EdgeStyle/.style={line width=1,->}}
	\Edge(a1)(a2)
	\Edge(a2)(a5)
    \Edge(a5)(a6)
	\Edge(a6)(a4)
	\Edge(a4)(a3)
    \Edge(a1)(a4)
	\Edge(a4)(a5)
	\Edge(a5)(a1)
	\Edge(a3)(a1)
  \tikzset{EdgeStyle/.style={bend right=25,line width=1,->}}
    \Edge(a2)(a3)
    \Edge(a6)(a2)
    \Edge(a3)(a6)
\end{tikzpicture}
\end{center}
\caption{$G_1$, a graph without the Riemann-Roch property}\label{fig::no_RR}
\end{figure}
It is easy to check that for $G_1$, $x_1=(1,0,0,1,0,2)$ and $x_2=(2,1,1,1,0,0)$ are both minimally non-terminating. Since $\deg(x_1)\neq \deg(x_2)$, Proposition \ref{prop::RR=>min_vegtln_deg=dist0} tells us that the Riemann--Roch formula does not hold for $G_1$. Note that $G_1$ is Eulerian, hence the Riemann--Roch property can also fail for Eulerian digraphs.

\subsubsection{Eulerian digraphs with non-natural Riemann--Roch property} 

From Theorem \ref{thm:natural_RR} if follows that we cannot have natural Rieman--Roch property if the graph is Eulerian and not bidirected.
A very simple example of an Eulerian digraph with the Riemann--Roch property is a directed cycle. It is straightforward that for a directed cycle, any chip-distribution of degree at least one is non-terminating. On the other hand, since $\minfas$=1, any chip-distribution of degree less than one is terminating.
Thus, the minimally non-terminating distributions are exactly the distributions of degree 1.
Let $C$ be any distribution of degree 2. 
It is straightforward that the conditions of Theorem \ref{thm::szuks_elegs_felt_alt_digraf} hold, thus a directed cycle has the Riemann--Roch property.

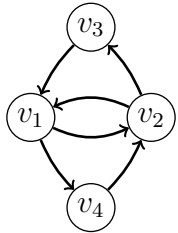
\begin{figure}[ht]
\begin{center}
\begin{tikzpicture}[scale=0.8]
  \GraphInit[vstyle=Dijkstra]
	\Vertex[x=0,y=0,L=$v_1$]{a1}    
	\Vertex[x=2,y=0,L=$v_2$]{a2}    
	\Vertex[x=1,y=1.5,L=$v_3$]{a3}
	\Vertex[x=1,y=-1.5,L=$v_4$]{a4}    
  \tikzset{EdgeStyle/.style={bend right=25,line width=1,->}}
    \Edge(a1)(a2)
    \Edge(a2)(a1)
  \tikzset{EdgeStyle/.style={bend right=10,line width=1,->}}
    \Edge(a3)(a1)
    \Edge(a2)(a3)
    \Edge(a1)(a4)
    \Edge(a4)(a2)
\end{tikzpicture}
\end{center}
\caption{$G_2$, a graph with non-natural Riemann--Roch property}\label{fig::non-natural_RR}
\end{figure}
Another example, where there are more than one equivalence classes with the same number of chips is the following graph $G_2$ (see also Figure \ref{fig::non-natural_RR}):

\[ 
V(G_2)=\{v_1,v_2,v_3,v_4\}; \quad 
E(G_2)=\{\overrightarrow{v_1v_2},\overrightarrow{v_1v_4},\overrightarrow{v_2v_1},\overrightarrow{v_2v_3},\overrightarrow{v_3v_1},\overrightarrow{v_4v_2}\}.
\]

For this graph, $\dist(\mathbf{0}_{G_2})=\minfas(G_2)=2$. It is well-known that for an Eulerian digraph, the number of equivalence classes of chip-distributions of a fixed degree equals the number of spanning in-arborescences rooted at an arbitrary vertex $v$ (this follows from the fact that the determinant of the reduced Laplacian matrix of an Eulerian digraph is equal to the number of spanning in-arborescences rooted at $v$, see \cite[Theorem 10.4]{Stanley}). Thus this graph has two equivalence classes of degree 2. It easy to check that the distribution $(2,0,0,0)$ is non-terminating, while the distribution $(1,1,0,0)$ is terminating. So for $C = (4,0,0,0)$ the conditions of Theorem \ref{thm::szuks_elegs_felt_alt_digraf} hold.

\subsubsection{A non-Eulerian digraph with natural Riemann--Roch property}\label{ss::example_natural}

We have seen in Section \ref{ss::natural_iff_bidirected} that an Eulerian digraph has the natural Riemann--Roch property if and only if it is bidirected. Here we show that there exist also non-Eulerian digraphs with the natural Riemann--Roch property. 

Let $G_3$ be the following digraph (see also Figure \ref{fig::natural_nonEulerian}).
\begin{eqnarray*}
V(G_3) &=& \{v_1,v_2,v_3,v_4\};\\
E(G_3) &=& \{\overrightarrow{v_1v_2},\overrightarrow{v_1v_4},\overrightarrow{v_2v_1},\overrightarrow{v_3v_2},\overrightarrow{v_3v_2},\overrightarrow{v_4v_3},\overrightarrow{v_4v_3},\overrightarrow{v_4v_1}\}.
\end{eqnarray*}

\begin{figure}[ht]
\begin{center}
\begin{tikzpicture}
  \GraphInit[vstyle=Dijkstra]
	\Vertex[x=0,y=0,L=$v_4$]{a1}    
	\Vertex[x=2,y=0,L=$v_3$]{a2}    
	\Vertex[x=2,y=2,L=$v_2$]{a3}
	\Vertex[x=0,y=2,L=$v_1$]{a4}   
  \tikzset{EdgeStyle/.style={bend left=25,line width=1,->}}
    \Edge(a1)(a2)
    \Edge(a2)(a3)
  \tikzset{EdgeStyle/.style={bend right=25,line width=1,->}}
    \Edge(a1)(a2)
    \Edge(a2)(a3)
    \Edge(a3)(a4)
    \Edge(a4)(a1)
    \Edge(a4)(a3)
    \Edge(a1)(a4)
\end{tikzpicture}
\end{center}
\caption{$G_3$, a non-Eulerian digraph with the natural Riemann--Roch property}\label{fig::natural_nonEulerian}
\end{figure}
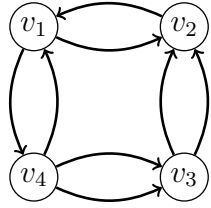

We claim that in this digraph, there is only one minimally non-terminating equivalence class, which is the equivalence class of $(1,0,0,3)$.

First, we show that a minimally non-terminating distribution needs to have degree 4. 
Note that $G_3$ is strongly connected. By Lemma \ref{lem::eof_mind_vegt_sokszor_lo}, in a non-terminating game on a strongly connected digraph, each vertex is fired infinitely often. Hence in a non-terminating game on $G_3$, there must be a time when each vertex has already fired, and $v_4$ can be fired at the moment.
Thus, each non-terminating equivalence class on $G_3$ contains an element with at least 0 chips on each vertex, and at least 3 chips on $v_4$. Hence a non-terminating chip-distribution needs to have at least 3 chips. Moreover, a non-terminating degree-3 equivalence class could only be the class of $(0,0,0,3)$, but it is easy to check that this distribution is terminating.

In a non-terminating equivalence class of degree 4, there is also an element $x$ with $x\geq (0,0,0,3)$. We have four choices for this: $(1,0,0,3),(0,1,0,3),(0,0,1,3)$ and $(0,0,0,4)$. Among these, $(1,0,0,3)\sim(0,1,0,3)\sim(0,0,0,4)$ and these are non-terminating, and $(0,0,1,3)$ is terminating. 

In the case if the degree of a chip-distribution $x$ is 5, and $x\geq (0,0,0,3)$: 
If $x-(0,0,0,3)$ has at least one chip on $v_1$ or on $v_2$ or on $v_4$, then $x\geq 
(1,0,0,3)$ or $x\geq (0,1,0,3)$ or $x\geq (0,0,0,4)$, hence it cannot be minimally non-terminating. The only remaining chip-distribution $x$ with $\deg(x)=5$ and $x\geq (0,0,0,3)$ is $(0,0,2,3)$. However, it is easy to check that $(0,0,2,2)$ is non-terminating, hence $(0,0,2,3)$ is also not minimally non-terminating.

We have $|E(G_3)|-|V(G_3)|+1=5$, hence by \cite{BL92}, each chip-distribution on $G_3$ with degree at least 5 is non-terminating. Hence a chip-distribution with degree at least 6 cannot be minimally non-terminating.

Since we only have one minimally non-terminating equivalence class, the conditions of Theorem \ref{thm::szuks_elegs_felt_alt_digraf} trivially hold with $C=(2,0,0,6)\sim (2,1,2,3)=\mathbf{d^+_{G_3}}$. Thus, $G_3$ has the natural Riemann--Roch property.

\section*{Acknowledgements}

Research was supported by the Hungarian National Research, Development and Innovation Office – NKFIH, grant no. K109240.
We would like to thank Darij Grinberg for helpful suggestions concerning the first version of the paper. 

\section*{\refname}

\bibliographystyle{abbrv}
\bibliography{chip_rr}

\begin{thebibliography}{10}

\bibitem{lattice}
O.~Amini and M.~Manjunath.
\newblock Riemann-{R}och for sub-lattices of the root lattice {$A\sb n$}.
\newblock {\em Electron. J. Combin.}, 17(1):R124, 2010.

\bibitem{asadi}
A.~Asadi and S.~Backman.
\newblock Chip-firing and {R}iemann-{R}och theory for directed graphs.
\newblock {\it Preprint}, \verb=http://arxiv.org/abs/1012.0287=, 2011.

\bibitem{BN-Riem-Roch}
M.~Baker and S.~Norine.
\newblock Riemann--{R}och and {A}bel--{J}acobi theory on a finite graph.
\newblock {\em Adv. Math.}, 215(2):766--788, 2007.

\bibitem{BL92}
A.~Bj{\"o}rner and L.~Lov{\'a}sz.
\newblock Chip-firing games on directed graphs.
\newblock {\em J. Algebraic Combin.}, 1(4):305--328, 1992.

\bibitem{BLS91}
A.~Bj{\"o}rner, L.~Lov{\'a}sz, and P.~W. Shor.
\newblock Chip-firing games on graphs.
\newblock {\em European J. Combin.}, 12(4):283--291, 1991.

\bibitem{Bond-Levine}
B.~Bond and L.~Levine.
\newblock Abelian networks {I}. {F}oundations and examples.
\newblock {\em SIAM J.~Discrete Math.}, 30(2):856--874., 2016.

\bibitem{cori}
R.~Cori and Y.~le~Borgne.
\newblock On computation of {B}aker and {N}orine’s rank on complete graphs.
\newblock {\em Electron. J. Combin.}, 23(1):P1.31, 2016.

\bibitem{farrell-levine-coeulerian}
M.~Farrell and L.~Levine.
\newblock Coeulerian graphs.
\newblock {\em Proc.~Amer.~Math.~Soc.}, 144:2847--2860, 2016.

\bibitem{Gallai}
T.~Gallai.
\newblock On directed paths and circuits.
\newblock In {\em Theory of Graphs (Proc.~Coll.~held at Tihany, Hungary,
  September, 1966)}, pages 115--118. Akad\'emiai Kiad\'o, Budapest, 1968.

\bibitem{trop_GK}
A.~Gathmann and M.~Kerber.
\newblock A {R}iemann-{R}och theorem in tropical geometry.
\newblock {\em Math. Zeitschrift}, 259:217--230, 2008.

\bibitem{godsil-royle}
C.~Godsil and G.~F. Royle.
\newblock {\em Algebraic Graph Theory}.
\newblock Springer-Verlag New York, 2001.

\bibitem{Hujter-Kiss-Tothmeresz.17.Reachability}
B.~Hujter, V.~Kiss, and L.~T\'othm\'er\'esz.
\newblock On the complexity of the chip-firing reachability problem.
\newblock {\em Proc.~Amer.~Math.~Soc.}, 145:3343--3356, 2017.

\bibitem{rang_NP_nehez}
V.~Kiss and L.~T\'othm\'er\'esz.
\newblock Chip-firing games on {E}ulerian digraphs and $\mathbf{NP}$-hardness
  of computing the rank of a divisor on a graph.
\newblock {\em Discrete Appl.\ Math.}, 193:48--56, 2015.

\bibitem{PPW14}
D.~Perkinson, J.~Perlman, and J.~Wilmes.
\newblock Primer for the algebraic geometry of sandpiles.
\newblock {\em Contemp.\ Math.}, 605, 2014.

\bibitem{Perrot}
K.~Perrot and T.~Van~Pham.
\newblock Feedback {A}rc {S}et {P}roblem and {NP}-{H}ardness of {M}inimum
  {R}ecurrent {C}onfiguration {P}roblem of {C}hip-{F}iring {G}ame on {D}irected
  {G}raphs.
\newblock {\em Ann. Comb.}, 19(1):1--24, 2015.

\bibitem{Stanley}
R.~Stanley.
\newblock {\em Algebraic Combinatorics}.
\newblock Springer-Verlag New York, 2013.

\bibitem{wilmes_bsc}
J.~S. Wilmes.
\newblock Algebraic invariants of sandpile graphs.
\newblock Bachelor's thesis, The Division of Mathematics and Natural Sciences,
  Reed College, 2010.

\end{thebibliography}

\end{document}